\documentclass{amsart}
\usepackage{mathrsfs,amssymb,mathrsfs, multirow,setspace}
\usepackage{amsmath}
\usepackage{listings}
\usepackage{pdfpages}
\usepackage{caption}
\usepackage[a4paper, total={7in, 9in}]{geometry}
\usepackage{enumerate}
\theoremstyle{plain}
\usepackage{graphicx}
\usepackage{mathtools}
\usepackage[utf8]{inputenc}

\newtheorem{theorem}{Theorem}[section]
\newtheorem{lemma}[theorem]{Lemma}

\theoremstyle{definition}

\newtheorem{example}[theorem]{Example}

\theoremstyle{remark}

\input xy 
\xyoption{all}

\begin{document}
	\title[On line upper ideal relation graphs of rings]{On line upper ideal relation graphs of rings}

   \author[Mohd Shariq, Praveen Mathil, Jitender Kumar]{ Mohd Shariq, Praveen Mathil, Jitender Kumar$^{*}$}
   \address{Department of Mathematics, Birla Institute of Technology and Science Pilani, Pilani-333031, India}
 \email{shariqamu90@gmail.com, maithilpraveen@gmail.com,  jitenderarora09@gmail.com}

\begin{abstract}

 The upper ideal relation graph $\Gamma_{U}(R)$ of a commutative ring $R$ with unity is a simple undirected graph with the set of all non-unit elements of $R$ as a vertex set and two vertices $x$, $y$ are adjacent if and only if the principal ideals $(x)$ and $(y)$ are contained in the principal ideal $(z)$ for some non-unit element $z\in R$. This manuscript characterizes all the Artinian rings $R$ such that the graph $\Gamma_{U}(R)$ is a line graph. Moreover, all the Artinian rings $R$ for which $\Gamma_{U}(R)$ is the complement of a line graph have been described.
 

\end{abstract}
\maketitle

\section{Historical Background and Preliminaries}

The study of graphs linked with algebraic structures is a significant area of research within the algebraic graph theory. This field connects algebra and graph theory, two essential branches of mathematics. Investigating graphs associated with algebraic structure is valuable due to their practical application as demonstrated by notable literature works (see\cite{kelarev2003graph,kelarev2009cayley,kelarev2004labelled}). Several graphs associated to rings including the cozero-divisor graph \cite{afkhami2011cozero}, zero-divisor graph \cite{Anderson1999zero}, annihilator graph \cite{badawi2014annihilator}, intersection graph of ideals \cite{chakrabarty2009intersection}, co-maximal graph \cite{maimani2008comaximal}, prime ideal sum graph \cite{saha2023prime}, upper ideal relation graph \cite{baloda2023study} of rings have been extensively studied in the literature. Line graph of a graph $\Gamma$ is a graph which is free from certain induced subgraphs  (see Theorem \ref{linegraphchar}). The line graph characterization of various graphs associated with algebraic structures has been extensively studied. For instance, all the groups with line power graphs (or line enhanced power graphs) have been ascertained in \cite{bera2022line,kumar2023finite}. All rings whose zero-divisor graphs and cozero-divisor graphs are line graphs have been determined in \cite{barati2021line,khojasteh2022line}.  Line graphs of unit graphs associated with the direct product of rings have been investigated in \cite{pirzada2022line}. Various graph-theoretic parameters of the line graph of the zero divisor graph and unit graph have been investigated in \cite{singh2022graph} and \cite{MR4543335}, respectively.
 
Baloda \textit{et al.} \cite{baloda2023study} introduced and studied the upper ideal relation graphs of rings.  They have investigated the minimum degree, metric dimension, perfectness, planarity, and the independence number of $\Gamma_U(R)$. The upper ideal relation graph $\Gamma_{U}(R)$ of a commutative ring $R$ with unity is a simple undirected graph with the set of all non-unit elements of $R$ as a vertex set and two vertices $x$, $y$ are adjacent if and only if the principal ideals $(x)$ and $(y)$ are contained in the principal ideal $(z)$ for some non-unit element $z\in R$. Further, embeddings of $\Gamma_U(R)$ on certain surfaces have been examined in \cite{Barkha}.

This manuscript investigates the line graph characterization of the upper ideal relation graph of a ring. Section \ref{section2} characterizes all the Artinian rings $R$ such that $\Gamma_{U}(R)$ is a line graph. Moreover, all Artinian rings $R$ such that $\Gamma_{U}(R)$ is the complement of a line graph have been described in Section \ref{section3}. Now, we recall the essential definitions and results. By \emph{graph} $\Gamma(V, E)$, we refer to an undirected simple graph with vertex set $V$ and edge set $E$.  A graph $\Gamma'(V', E')$ is a \emph{subgraph}  of $\Gamma$ if and only if $V'\subseteq V $ and $E'\subseteq E $. The complement  $\overline{\Gamma}$ of a graph $\Gamma$ such that $V(\overline{\Gamma})=V(\Gamma)$  and ${x,y}\in V(\overline{\Gamma})$ are adjacent if and only if they are not adjacent in $\Gamma$.


The subgraph induced by the set $X \subseteq V(\Gamma)$, denoted by $\Gamma(X)$, is the graph such that  $V(\Gamma(X))=X$ and $x,y$ are adjacent in $\Gamma(X)$ if and only if they are adjacent in $\Gamma$. A \emph{path} in a graph is a sequence of vertices where each pair of consecutive vertices is connected by an edge. The path graph $P_n$ is a path on $n$ vertices. A graph $\Gamma$ is labelled \emph{complete} if every pair of vertices is joined in $\Gamma$. We use $K_n$ to denote the complete graph with $n$ vertices. By $mK_n$ means $m$ copies of $K_n$. A \emph{bipartite} graph $\Gamma$ is a graph whose vertex set can be partitioned into two subsets such that no two vertices in the same partition subset are adjacent. The \emph{complete bipartite graph} with partition size $p$ and $q$ is denoted by $K_{p, q}$. 

If the ring $R$ has exactly one maximal ideal then $R$ is called a \emph{local ring}. By the structural theorem of Artinian ring \cite{atiyah1969introduction}, a commutative Artinian ring $R$ is isomorphic to the direct product of local rings $R_i$ that is $R \cong R_1 \times R_2 \times \cdots \times R_n$. Throughout the paper, $F_i$ is a field. The ideal generated by $x\in R$ is denoted by $(x)$. The set of all the unit elements and zero divisors of the ring $R$ is denoted by $U(R)$ and $Z(R)$, respectively. The \emph{characteristic} $char(R)$ of a ring is the least positive integer $n$ such that $nx=0$ for all $x\in R$. Other basic results and definitions on rings can be referred to  \cite{atiyah1969introduction}. The line graph $L(\Gamma)$ of the graph $\Gamma$ such that $V(L(\Gamma))=E(\Gamma)$ and its vertices are adjacent if they have a common vertex in $\Gamma$. The following descriptions of the line graphs will be used often (see \cite{beineke1970characterizations}).  

 \begin{theorem}\label{linegraphchar} 
  A graph is a line graph of some graph if and only if it is free from the nine induced subgraphs $\Gamma_i$$(1\leq i\leq 9)$ given in  \rm{Figure \ref{forbiddengraphs}}.
    \begin{figure}[h!]
			\centering
			\includegraphics[width=1 \textwidth]{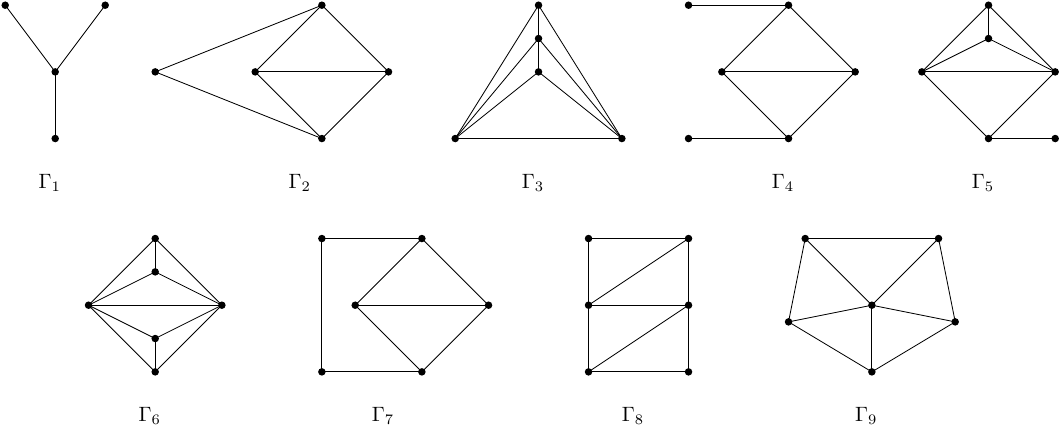}
			\caption{Forbidden induced subgraphs of line graph }
   \label{forbiddengraphs}
\end{figure}
 \end{theorem}

  \begin{theorem}\label{complement}
  A graph is the complement of a line graph if and only if it is free from the nine induced subgraphs $\overline{\Gamma_i}$ $(1\leq i\leq9)$ given in \rm{Figure \ref{complenentforbiddengraphs}}. 
    \begin{figure}[h!]
			\centering
			\includegraphics[width=1 \textwidth]{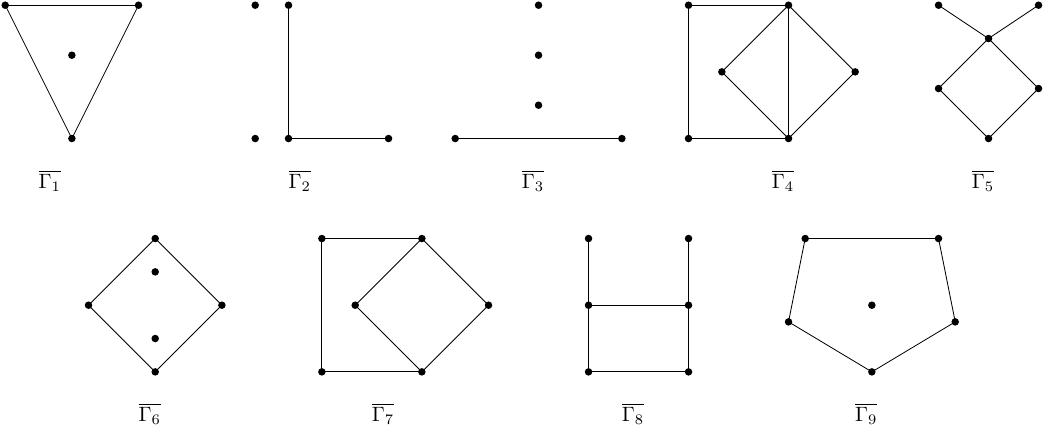}
			\caption{Forbidden induced subgraphs of complement of line graphs }
   \label{complenentforbiddengraphs}
\end{figure}
 \end{theorem}

\section{Characterization of rings $R$ such that $\Gamma_U(R)$ is a line graph }\label{section2}

 In this section, we shall classify all the Artinian rings whose upper ideal relation graphs are line graphs. We begin with a characterization of non-local rings $R$ such that $\Gamma_U(R)$ is a line graph.


 \begin{theorem}\label{fourproduct}
Let $R$ be a non-local ring and let $F_1$, $F_2$ be fields. Then $\Gamma_U(R)$ is a line graph of some graph if and only if $R\cong F_1\times F_2$.

  \end{theorem}
  \begin{proof}
Since $R$ is a non-local commutative ring, we have  $R \cong R_1 \times R_2 \times \cdots \times R_n$, where $n\geq 2$ and each $R_i$ is a local ring. First assume that $\Gamma_U(R)$ is a line graph. Let $n\geq 3$. For the set $A = \{$(0,0,0,\ldots,0)$,$(0,1,1,\ldots,1)$, $(1,0,1,\ldots,1)$,$(1,1,0,\ldots,1)$\}$, notice that the subgraph of $\Gamma_U(A) \cong \Gamma_1$, a contradiction. Therefore, we get $R\cong R_1\times R_2$. If $R_1$ and $R_2$ are not fields, then $|Z(R_1)|\ge2$ and $|Z(R_2)|\ge2$. Let $a\in Z(R_1)\setminus\{0\}$ and $b\in Z(R_2)\setminus\{0\}$. For the set $S=\{(a,0),(0,b),(0,0),(a,1),(1,b)\}$, we have   $\Gamma_U(S)\cong\Gamma_3$. Consequently, $\Gamma_U(R_1\times R_2)$ can not be a line graph.
 
 We may now suppose that one of the rings $R_1$ or $R_2$ is a field. With no loss of generality, assume that $R_2$ is not a field, Therefore, $|U(R_2)|\ge2$ and $|Z(R_2)|\ge2$. Let $a\in Z(R_2)\setminus\{0\}$, $u\in U(R_2)\setminus\{1\}$. Consider the set $Y=\{( 1,0), ( 1,a),( 0,0),( 0,a),( 0,1),( 0,u)\}$. Note that  $\Gamma_U(Y)\cong \Gamma_6$, which is not possible. Thus, $R\cong F_1\times F_2$.

 Conversely, let $R\cong F_1\times F_2$. Recall that $\Gamma_U(F_1\times F_2)=K_1\bigvee(K_{|F_1|-1}\bigcup K_{|F_2|-1})$ (see \cite{baloda2023study}). Then a subgraph $\Gamma$ of $\Gamma_U(F_1\times F_2)$  induced by any subset of $V(\Gamma_U(F_1\times F_2))\setminus \{(0,0)\}$ is either a complete graph or a disjoint union of complete graphs. It follows that $\Gamma \not \cong \Gamma_i$, where $1\leq i\leq9$. Further, for a subgraph $\Gamma'$ of $\Gamma_U(F_1\times F_2)$ induced by a set of vertices including (0,0), $\Gamma'$ can be isomorphic to one of the subgraphs $\Gamma_1$, $\Gamma_3$ and $\Gamma_9$. But this is also not possible because none of the graph $\Gamma_1$, $\Gamma_3$, $\Gamma_9$ is isomorphic to the graph $K_1\bigvee(K_m\bigcup K_n)$.
   Therefore, there always exists a graph whose line graph is the graph $\Gamma_U(F_1\times F_2)$. This completes our proof.
   \end{proof} 
  \begin{example}
     Let $R\cong  \mathbb{Z}_3\times\mathbb{Z}_5$. Then observe that $\Gamma_U(\mathbb{Z}_3\times\mathbb{Z}_5)=L(K)$ (see Figure 3).
      \end{example}
 \begin{figure}[h!]
			\centering
			\includegraphics[width=0.8 \textwidth]{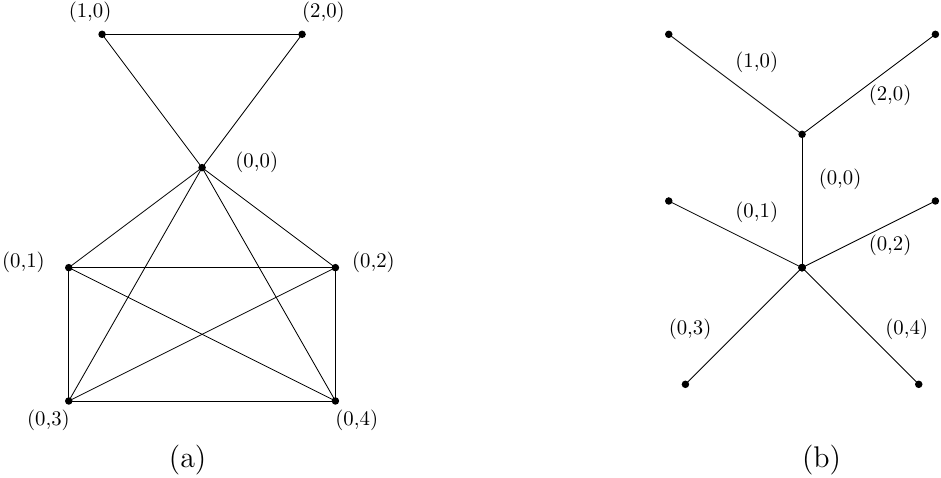}
		\caption{ $\textbf{(a)}$  ${\Gamma_U(\mathbb{Z}_3\times \mathbb{Z}_5)}$ $\hspace{.05cm}$ $\textbf{(b)}$ $K$ }
    \label{figexample.pd}
 \end{figure}
 
 Next, we investigate Artinian local commutative rings such that their upper ideal relation graphs are line graphs of some graph.
 \begin{theorem}
The upper ideal relation graph of a local ring $R$ is a line graph of some graph if and only if $R$ is a principal ideal ring. 
 \end{theorem}
 \begin{proof}
 Let $\Gamma_U(R)$ be a line graph of some graph. On the contrary, suppose that $R$ is not a principal ideal ring. 
 Let $\mathcal{M}$ be a maximal ideal of $R$  with a minimal generating set $\{x_1, x_2, \ldots, x_n\}$, where ($n\geq2$). Observe that the subgraph $\Gamma_U(A)$, where $A=\{x_1, x_2, x_1+x_2, 0\}$, is isomorphic to $\Gamma_1$, which is not possible. Thus, $R$ is a principal ideal ring.
 Conversely, assume that $R$ is a principal ideal ring. Since $R$ is a local ring, the graph $\Gamma_U(R)$ is complete and so is a line graph of a star graph.
 \end{proof}

 \section{Charecterization of rings $R$ such that $\Gamma_U(R)$ is the complement of a line graphs }\label{section3}
 In this section, we characterize all commutative Artinian rings whose upper ideal relation graph is the complement of a line graph. We begin with a classification of the commutative non-local rings $R$ such that $\Gamma_U(R)$ is the complement of a line graph.


 \begin{theorem}

      Let $R$ be a non-local commutative ring. Then $\Gamma_U(R)$ is the complement of a line graph if and only if $R$ is isomorphic to one of the following six rings: \begin{center}
          $\mathbb{Z}_2\times\mathbb{Z}_2\times\mathbb{Z}_2$,  $\mathbb{Z}_2\times\mathbb{Z}_2$, $\mathbb{Z}_2\times \mathbb{Z}_3$, $\mathbb{Z}_2\times \mathbb{Z}_4$, $\mathbb{Z}_2\times \frac{\mathbb{Z}_2[x]}{(x^2)}$, $\mathbb{Z}_3\times \mathbb{Z}_3$.
      \end{center}

 \begin{proof}
 Since $R$ is a non-local ring, we have  $R \cong R_1 \times R_2 \times \cdots \times R_n$, where $n\geq 2$ and each $R_i$ is a local ring.
     Let $\Gamma_U(R)$ be the complement of a line graph and let $n\geq 4$.
     Consider the set of vertices $A=\{x_1,x_2,x_3,x_4\}$ of $\Gamma_U(R)$, where $x_1=(1,0,0,1,\ldots,0)$, $x_2=(1,0,1,0,\ldots,0)$, $x_3=(1,1,0,0,\ldots,0)$, $x_4=(0,1,1,1,\ldots,0)$. Then the subgraph  $\Gamma_U(A)\cong\overline{\Gamma_1}$, a contradiction. Consequently, either $R \cong R_1 \times R_2 \times R_3$ or $R\cong R_1 \times R_2$. 
     
     First let $R \cong R_1 \times R_2 \times R_3$. Suppose that one of $R_i$,  where $i\in\{1,2,3\}$, is not a field. With no loss of generality, assume that $R_3$ is not a field. Then $|U(R_3)|\ge2$  and $|Z(R_3)|\ge2$. For ${u_1}\in {U(R_3)\setminus\{1\}}$  and $a\in Z(R_3)\setminus\{0\}$, consider the set $S=\{ (0,1,a), (0,1,u_1),(0,1,1),(1,0,u_1)\}$. Then the subgraph  $\Gamma_U(S)\cong\overline{\Gamma_1}$, a contradiction. Consequently, $R\cong R_1 \times R_2 \times R_3$, where each $R_i$ is a field. Further, suppose that the cardinality of one of the fields $R_i$ is at least $3$. With no loss of generality, assume that $|R_3|\ge3$. For ${u}\in {U(R_3)\setminus\{1\}}$, and the set $A=\{(0,0,1), (0,0,u),(0,1,u),(1,1,0)\}$, notice that $\Gamma_U(A)\cong \overline{\Gamma_1}$. Therefore, $\Gamma_U( R_1 \times R_2 \times R_3)$ can not be the complement of a line graph. It implies that $|R_i|\leq2$ for each $i$. Consequently, $R\cong\mathbb{Z}_2\times\mathbb{Z}_2\times\mathbb{Z}_2$.

Next let $R \cong R_1 \times R_2$. To prove our result, we consider the following cases: 
 
 \textbf{Case 1.} \textit{Both $R_1$ and $R_2$ are not fields}. Consider the set $S=\{(u,0),(1,0), (u,a), (0,1)\}$, where $u\in U(R_1)\setminus\{1\}$ and $a\in Z(R_2)\setminus\{0\}$. Then the subgraph $\Gamma_U(S)\cong\overline{\Gamma_1}$, and so $\Gamma_U( R_1 \times R_2 )$ can not be the complement of a line graph.

 \textbf{Case 2.} \textit{One of $R_i$, where $i\in\{1,2\}$, is a field}. With no loss of generality, assume that $R_1$ is a field but $R_2$ is not a field. Let $|R_1|\ge3$. Consider the set $T=\{(1,0),(u,a),(0,1),(u,0)\}$, where $u\in U(R_1)\setminus\{1\}$ and $a\in Z(R_2)\setminus\{0\}$. Then the subgraph $\Gamma_U(T)\cong\overline{\Gamma_1}$. It follows that $\Gamma_U( R_1 \times R_2 )$ can not be the complement of a line graph. Therefore $|R_1|\leq2$. Now let  $|R_2|>4$. Then $|U(R_2)|\geq 4$. Consider the set $X=\{(1,0),(0,u_1),(0,u_2),(0,1)\}$, where $u_1,u_2 \in U(R_2)\setminus\{1\}$. Then, the subgraph $\Gamma_U(X)\cong\overline{\Gamma_1}$. Thus, $\Gamma_U( R_1 \times R_2 )$ is not the complement of a line graph. Consequently, $|R_2|\leq4$, $|R_1|\leq2$ and so either $R\cong \mathbb{Z}_2\times \mathbb{Z}_4$ or $R\cong \mathbb{Z}_2\times \frac{\mathbb{Z}_2[x]}{(x^2)}$.

 \textbf{Case 3.} \textit{ Both $R_1$ and $R_2$ are the fields}. Suppose that $|R_i|\ge4$, for some $i$. With no loss of generality, assume that $|R_2|\geq4$. Consider the set $S=\{(1,0),(0,u_1), (0,u_2),(0,1)\}$,  where $u_1,u_2 \in U(R_2)\setminus\{1\}$. Then the subgraph $\Gamma_U(S)\cong\overline{\Gamma_1}$. Thus, $\Gamma_U( R_1 \times R_2 )$ can not be the complement of a line graph for $|R_i|\ge4$. Therefore, $R \cong R_1 \times R_2$, where $|R_i|\leq3$ for each $i$. Consequently, $R$ is isomorphic to one of the three rings:  $\mathbb{Z}_2\times\mathbb{Z}_2$, $\mathbb{Z}_2\times \mathbb{Z}_3$, $\mathbb{Z}_3\times \mathbb{Z}_3$.

  \begin{figure}[h!]
 			\centering
 			\includegraphics[width=1 \textwidth]{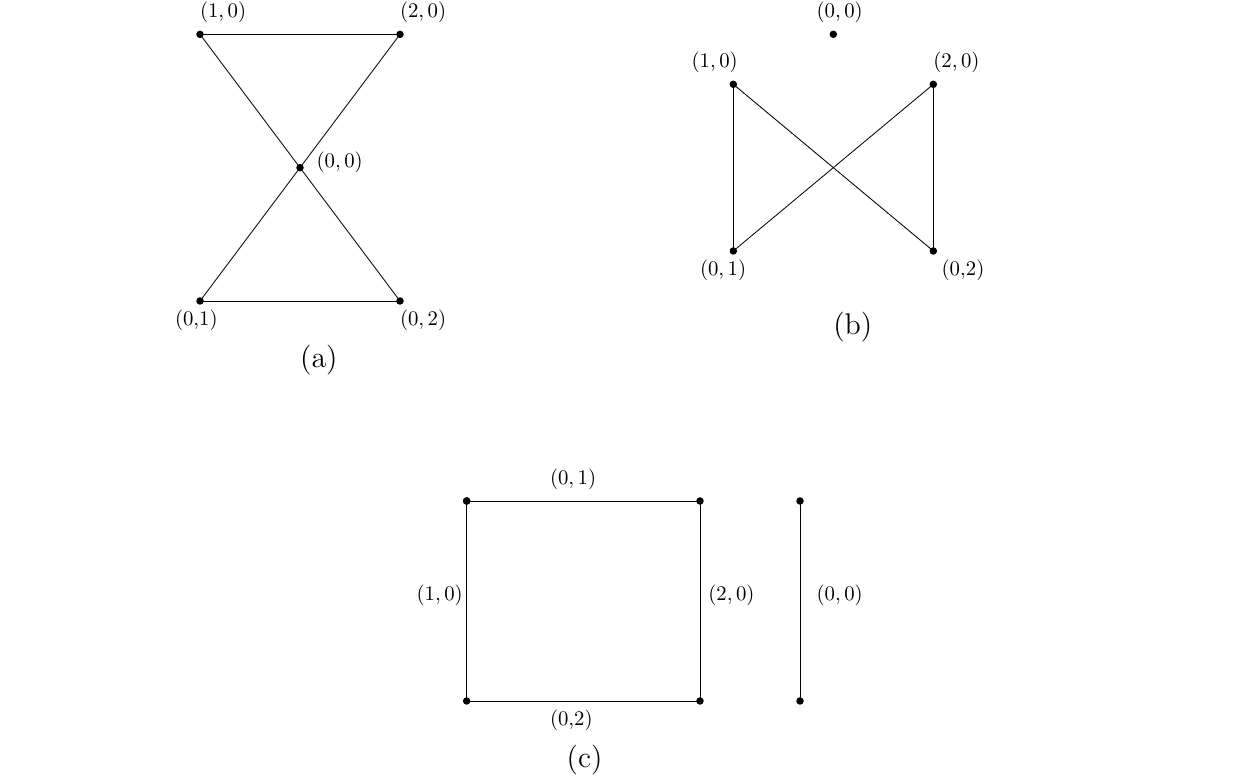}
 			\caption{ $\textbf{(a)}$ $\Gamma_U(\mathbb{Z}_3\times \mathbb{Z}_3)$  $\hspace{0.2cm}$  $\textbf{(b)}$ $\overline{\Gamma_U(\mathbb{Z}_3\times \mathbb{Z}_3)}$  $\hspace{0.2cm}$ $\textbf{(c)}$  $H_1$}
    \label{fig-5_graphs}
 \end{figure}

    \begin{figure}[h!]
			\centering
			\includegraphics[width=1 \textwidth]{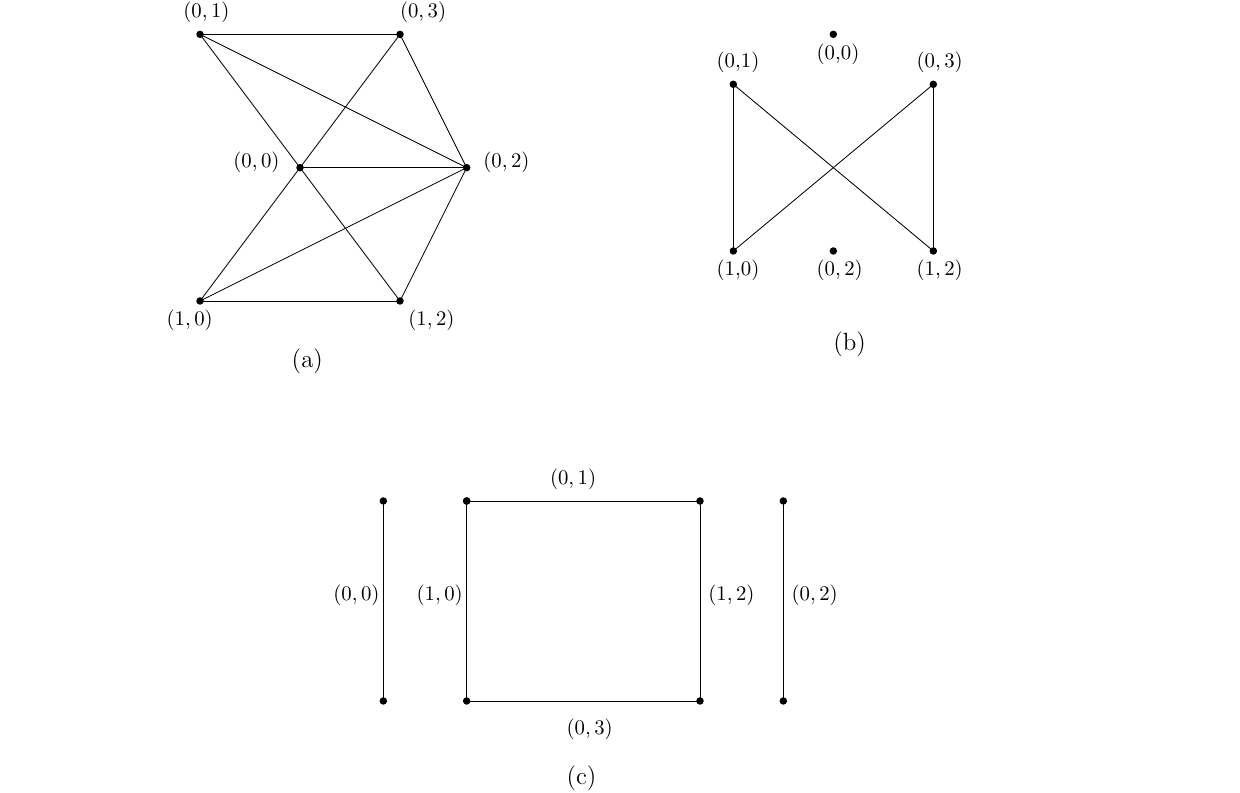}
 			\caption{ $\textbf{(a)}$ $\Gamma_U(\mathbb{Z}_2\times \mathbb{Z}_4)$ $\hspace{0.2cm}$  $\textbf{(b)}$  $\overline{\Gamma_U(\mathbb{Z}_2\times \mathbb{Z}_4)}$ $\hspace{.2cm}$ $\textbf{(c)}$ $H_2$ }
    \label{fig_6_graph}
 \end{figure}
  \begin{figure}[h!]
 			\centering
 			\includegraphics[width=1 \textwidth]{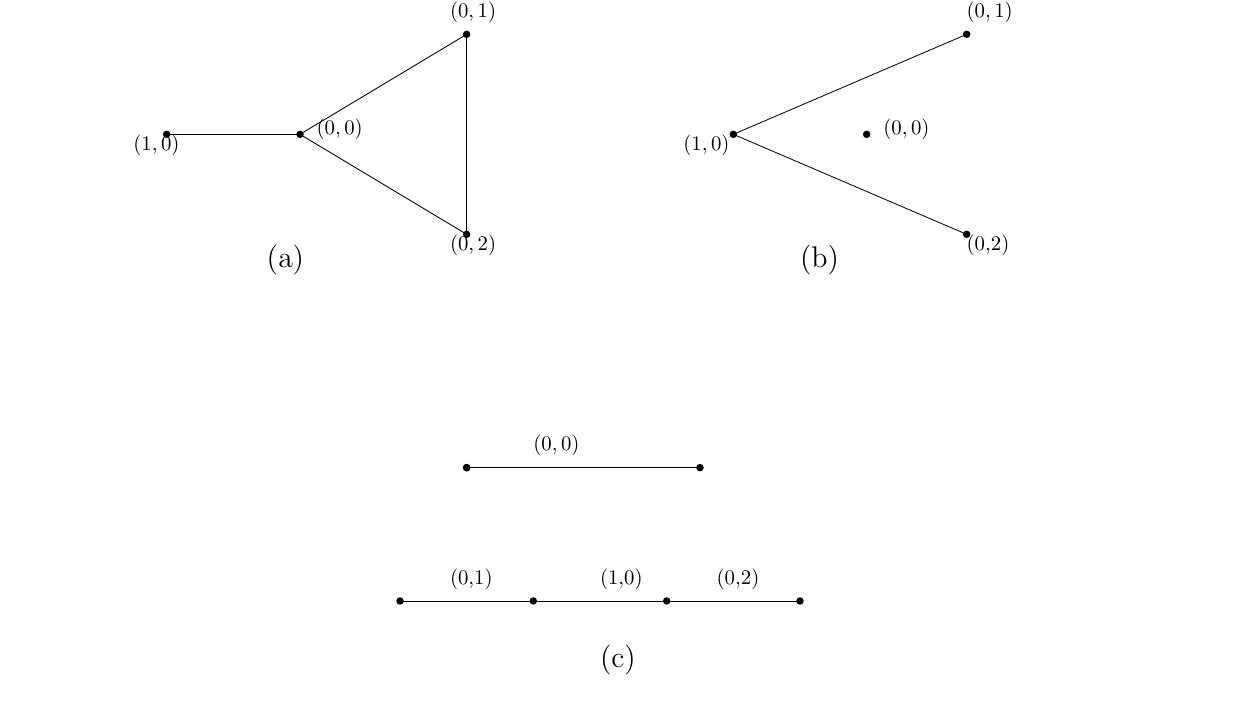}
 			\caption{ $\textbf{(a)}$ $\Gamma_U(\mathbb{Z}_2\times \mathbb{Z}_3)$ $\hspace{0.2cm}$  $\textbf{(b)}$  $\overline{\Gamma_U(\mathbb{Z}_2\times \mathbb{Z}_3)}$ $\hspace{.2cm}$ $\textbf{(c)}$ $H_3$ }
    \label{fig-7_graphs}
 \end{figure}

 Conversely, if $R\cong\mathbb{Z}_2\times\mathbb{Z}_2\times\mathbb{Z}_2$, then by Figure 4, note that  $\overline{\Gamma_U( \mathbb{Z}_2 \times \mathbb{Z}_2 \times \mathbb{Z}_2)}=L(H)$.
 \begin{figure}[h!]
			\centering
 			\includegraphics[width=1 \textwidth]{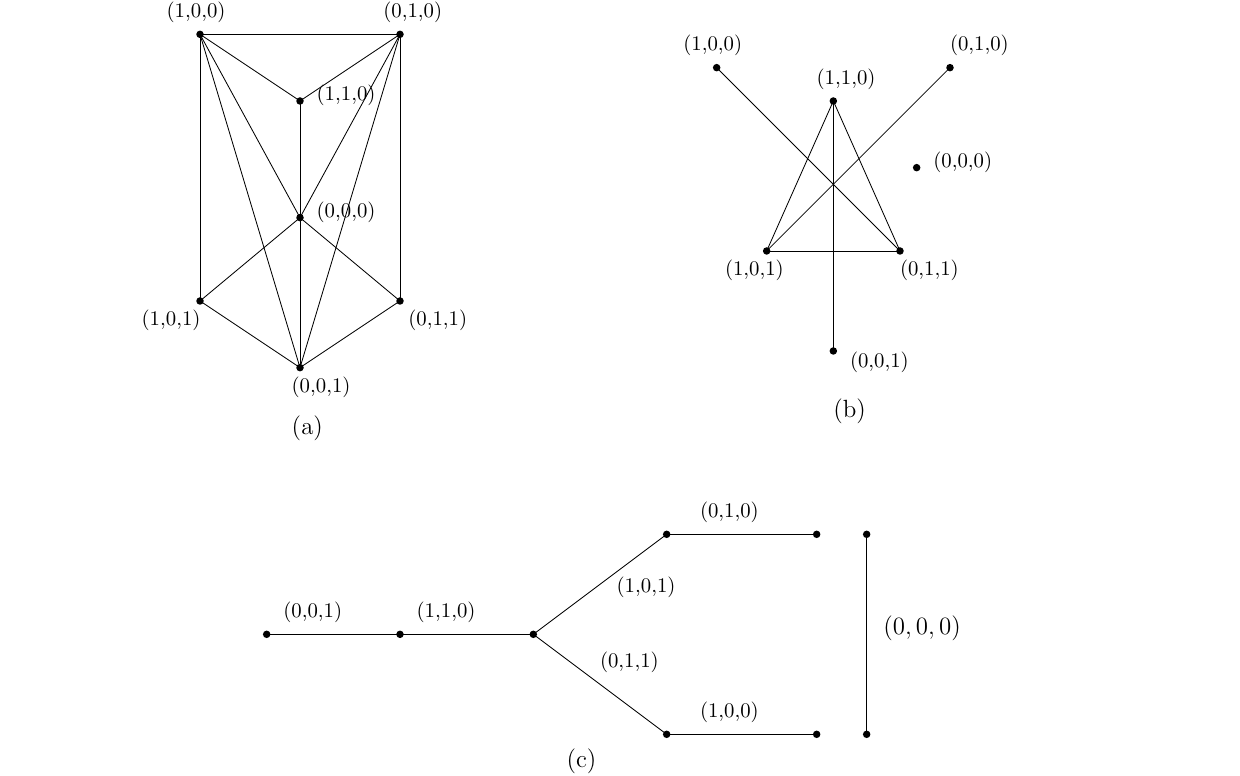}
 			\caption{$\textbf{(a)}$   $\Gamma_U(\mathbb{Z}_2\times \mathbb{Z}_2\times \mathbb{Z}_2)$ $\hspace{0.2cm}$ $\textbf{(b)}$ $\overline{\Gamma_U(\mathbb{Z}_2\times \mathbb{Z}_2\times \mathbb{Z}_2)}$  $\hspace{0.2cm}$  $\textbf{(c)}$  $H$ }
    \label{fig_4-graphs}
 \end{figure}
 By Figure 5, note  that $\Gamma_U(\mathbb{Z}_3\times \mathbb{Z}_3)=\overline{L(H_1)}$. Note that  $\Gamma_U(\mathbb{Z}_2\times \mathbb{Z}_4)$=$\Gamma_U(\mathbb{Z}_2\times \frac{\mathbb{Z}_2[x]}{(x^2)})$. By Figure 6, we obtain  $\Gamma_U(\mathbb{Z}_2\times \mathbb{Z}_4)=\overline{L(H_2)}$. By  Figure 7, observe that $\Gamma_U(\mathbb{Z}_2\times \mathbb{Z}_3)=\overline{L(H_3)}$. Also,  $\Gamma_U(\mathbb{Z}_2\times \mathbb{Z}_2)=P_3= \overline{L(P_2\bigcup P_3)}$ 
    \end{proof}
    \end{theorem}
 Now, we investigate local Artinian  rings $R$ such that $\Gamma_U(R)$ is the complement of a line graph.

 \begin{lemma}
     Let $R$ be a principal local ring. Then $\Gamma_U(R)$ is the complement of a line graph of some graph.
 \end{lemma}
 \begin{proof}
     Since $R$ is a principal local ring, it follows that  $\Gamma_U(R)$ is a complete graph and so $\Gamma_U(R)=\overline{L(nK_2)}$.
     \end{proof}

 \begin{lemma}
 Let $R$ be a non-principal local ring such that $ char(R)=p^k$, where $p>2$ is a prime. Then $\Gamma_U(R)$ is not the complement of a line graph of some graph.
 \end{lemma}
 \begin{proof}
  Let $\mathcal{M}$ be the maximal ideal of $R$ and let $\{ x_1,x_2,\ldots, x_n\}$, where $n \ge2 $, be its minimal generating set. Since $\textnormal{char}(R)= p^k$, there exists $x_i\in \mathcal{M}$ such that $2x_i\neq0$. With no loss of generality, assume that $2x_1\neq0$. Observe that for the set $S=\{x_1, 2x_1,x_1+x_2,x_2+2x_1,x_2 \}$, we have $\Gamma_U(S)\cong\overline{\Gamma_3}$. Therefore, by Theorem \ref{complement}, $\Gamma_U(R)$ can not be the complement of a line graph of some graph.
 \end{proof}

 \begin{lemma}
     Let $R$ be a non-principal local ring of characteristic $2^k (k\ge3)$ or $ 0$. Then $\Gamma_U(R)$ is not the complement of a line graph of some graph.
     \end{lemma}
 \begin{proof}
 Let $\{ x_1,x_2,\ldots, x_n\}$, where $n \ge2 $, be the minimal generating set of the maximal ideal $\mathcal{M}$ of $R$. Since $char=2^k$ (or $0$), then $R$ contains a subring $S=\{kI_R  \mid  k\in \mathbb Z\}$, where $I_R$ be the unity of ring $R$, which is isomorphic to the ring $\mathbb{Z}_{2^k}$ (or $\mathbb Z$). Let $J$ be a maximal ideal of $S$ such that $J=(2I_R)$. Consider $i\in\{1,2,\ldots,n\}$ such that $x_i\neq 2I_R$. Note that for the set  $A=\{2I_R,4I_R,6I_R,x_i\}$, we have  $\Gamma_U(A)\cong\overline{\Gamma_1}$. Therefore, $\Gamma_U(R)$ can not be the complement of a line graph.
     \end{proof}

 \begin{theorem}\label{char=4}
      Let $R$ be a non-principal local ring such that  $char(R)=4$ and let $\mathcal{M}$ be the maximal ideal of $R$ with the minimal generating set $\{x_1,x_2,\ldots,x_n\}$, where $n\geq2$. Then $\Gamma_U(R)$ is the complement of a line graph of some graph if and only if the following hold.
     \begin{enumerate}[\rm(i)]
         \item $\frac{R}{\mathcal{M}}\cong \mathbb{Z}_2$,
         \item $\mathcal{M}^2=\{0\}$,
         \item $2x_i=0$ for all $i$.
     \end{enumerate}
 \end{theorem}
 \begin{proof}
     Suppose that $\Gamma_U(R)$ is the complement of a line graph. 
 First let  $|{\frac{R}{\mathcal{M}}}|\geq3$. Consider the set  $A=\{x_1,ux_1,x_2,x_1+x_2,x_2+ux_1\}$, where $u\in U(R)$ such that $1-u\not\in \mathcal{M}$. Then the subgraph  $\Gamma_U(A)\cong\overline{\Gamma_3}$. It implies that $\Gamma_U(R)$ can not be the complement of a line graph. If $\mathcal{M}^2\neq\{0\}$, then there exist some $x_i, x_j$ such that either ${x_i}^2\neq0$ or $x_ix_j\neq0$. With no loss of generality, assume that ${x_1}^2\neq0$. Then for the set $B=\{x_1,{x_1}^2,x_1+x_2,x_2+{x_1}^2, x_2\}$, we have $\Gamma_U(B)\cong\overline{\Gamma_3}$.
 With no loss of generality, we may now suppose that $x_1x_2\neq0$. Since $char(R)=4$, we have $2+x \ne x$ for all $x$. Now, consider the set $S=\{x_1,x_1+x_1x_2,x_2,2+x_1,x_1+x_2\}$. Then the subgraph  $\Gamma_U(S)\cong\overline{\Gamma_3}$. Consequently, $\Gamma_U(R)$ can not be the complement of a line graph. Thus, $\mathcal{M}^2=\{0\}$. If  $2x_i\neq0$ for some $x_i\in \mathcal{M}$. With no loss of generality, assume that $2x_1\neq0$. Consider the set $X=\{x_1,2x_1,x_2,x_1+x_2,x_2+2x_1\}$. Then the subgraph  $\Gamma_U(X)\cong\overline{\Gamma_3}$, and so  $\Gamma_U(R)$ can not be the complement of a line graph. Consequently,  $2x_i=0$ for all $i$. Conversely, suppose that $(\rm{i})$,$(\rm{ii})$ and $(\rm{iii})$ holds. Since $|\frac{R}{\mathcal{M}}|=2$, we have $R={(0+\mathcal{M})}\bigcup{(1+\mathcal{M})}$. For $x\in{\mathcal{M}} $, we have $1{+x}\in U(R)$. It follows that for $v\in U(R)$, $v=1+y$ for some $y\in{\mathcal{M}}$. Since ${\mathcal{M}}^2=\{0\}$, note that $vx=(1+y)x=x+xy=x$. Therefore, the vertex set of  $\Gamma_U(R)$ consists the elements of the form $\sum_{i=1}^{n} {\alpha_ix_i} $ and $2+\sum_{i=1}^{n} {\alpha_ix_i} $, where $\alpha_i\in \{0,1\}$ only. Observe that any two non-zero vertices of $\Gamma_U(R)$ are not adjacent and the element zero of $R$  is adjacent to every other vertex of $\Gamma_U(R)$. Hence, $\Gamma_U(R)$ is a star graph and so $\Gamma_U(R)=\overline{L(K_2\bigcup K_{1,|V(\Gamma_U(R))|-1}}).$
 \end{proof}

 \begin{theorem}\label{char=2}
 Let $R$ be a non-principal local ring such that  $char(R)=2$ and let $\mathcal{M}$ be the maximal ideal of $R$ with the minimal generating set $\{x_1,x_2,\ldots,x_n\}$, where $n\geq2$ Then $\Gamma_U(R)$ is the complement of a line graph if and only if $\frac{R}{\mathcal{M}}\cong \mathbb{Z}_2$ and either $\mathcal{M}^2=0$ or $\mathcal{M}=\langle x_1, x_2\rangle$ such that ${x_1}^2={x_2}^2=0$ and $x_1x_2\neq0 $. 
 \end{theorem}
 \begin{proof}
 Suppose that $\Gamma_U(R)$ is the complement of a line graph. First let  $|{\frac{R}{\mathcal{M}}}|\geq3$. Consider the set  $A=\{x_1,ux_1,x_2,x_1+x_2,x_2+ux_1\}$, where $u\in U(R)\setminus\{1\}$ such that $1-u\not\in \mathcal{M}$. Observe that  $\Gamma_U(A)\cong\overline{\Gamma_3}$. It implies that  $\Gamma_U(R)$ can not be the complement of a line graph. Therefore, $|{\frac{R}{\mathcal{M}}}|=2$. If $\mathcal{M}^2=\{0\}$, then there is nothing to prove.

 Let $\mathcal{M}^2\neq\{0\}$. If $n\ge3$,  then for the set $S=\{x_1,x_1+x_1x_2,x_2,x_3,x_2+x_3\}$, we have $\Gamma_U(S)\cong\overline{\Gamma_3}$. It follows that  $\Gamma_U(R)$ can not be the complement of a line graph. Therefore, $n=2$. Suppose there exists some $x_i$ such that ${x_i}^2\neq0$. With no loss of generality, assume that ${x_1}^2\neq0$. For the set $T=\{x_1,{x_1}^2,x_1+x_2,x_2+{x_1}^2, x_2\}$, we obtain $\Gamma_U(T)\cong\overline {\Gamma_2}$, if $x_1x_2=0$ or  $\Gamma_U(T)\cong\overline{\Gamma_3}$, if $x_2x_2\neq0$. Consequently, $\Gamma_U(R)$ can not be the complement of a line graph. Thus, ${x_1}^2={x_2}^2=0$ and $x_1x_2\neq0$. 

 Conversely, first suppose that ${\mathcal{M}}^2=\{0\}$ and $\frac{R}{\mathcal{M}}\cong \mathbb Z_2$.
 Then, by the converse part of the Theorem \ref{char=4}, we have  $\Gamma_U(R)$ is a star graph and so is the complement of a line graph of $K_2\bigcup K_{1,|V(\Gamma_U(R))|-1}$. We may now assume that $\frac{R}{\mathcal{M}}\cong \mathbb Z_2$ and $\mathcal{M}=\langle x_1, x_2\rangle$ with ${x_1}^2={x_2}^2=0$, $x_1x_2\neq0 $. Note that  \[V(\Gamma_U(R))=\mathcal{M}=\{\alpha_1x_1+\alpha_2x_2+\alpha_3x_1x_2 \mid  \alpha_i\in U(R)\cup \{0\}\}.\] 
 Moreover, $\alpha_i=1+y$ for $y\in\mathcal{M}$. It follows that $V(\Gamma_U(R))=\{0,x_1,x_2,x_1x_2, x_1+x_2,x_1+x_1x_2, x_2+x_1x_2,x_1+x_2+x_1x_2\}$. By Figure 8,  $\Gamma_U(R)=\overline{L(K)}$. 
  \begin{figure}[h!]
 			\centering
 			\includegraphics[width=1 \textwidth]{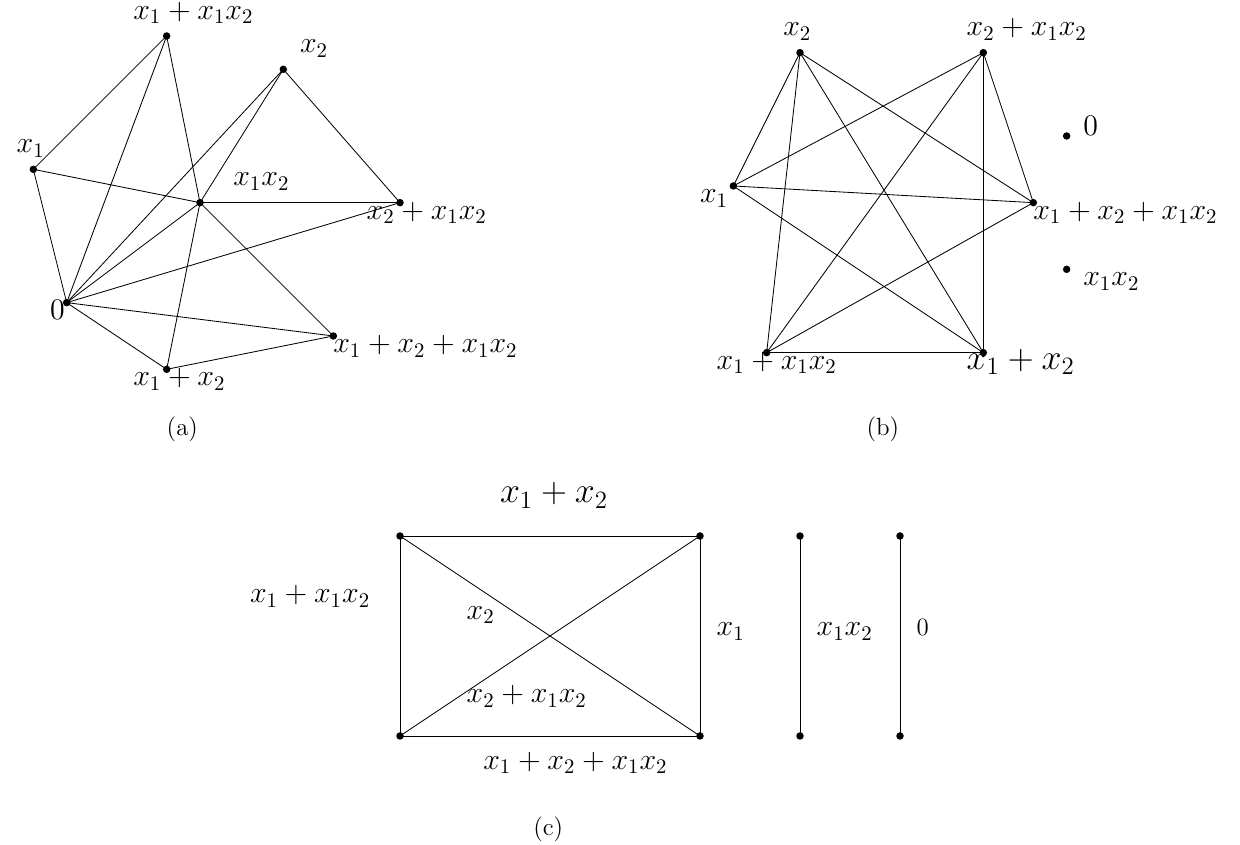}
 			\caption{ $\textbf{(a)}$ $\Gamma_U(R)$ $\hspace{0.2cm}$  $\textbf{(b)}$  $\overline{\Gamma_U(R)}$ $\hspace{.2cm}$ $\textbf{(c)}$ $K$ }
    \label{fig-8_graphs}
 \end{figure}
  
 \end{proof}

 \textbf{Acknowledgement:} The first and second authors extend their gratitude to the Birla Institute of Technology and Science (BITS) Pilani, India, for providing financial support.

 \vspace{.3cm}

 \textbf{Conflicts of interest/Competing interests}: There is no conflict of interest regarding the publishing of this paper.




    


\begin{thebibliography}{10}

\bibitem{afkhami2011cozero}
M.~Afkhami and K.~Khashyarmanesh.
\newblock The cozero-divisor graph of a commutative ring.
\newblock {\em Southeast Asian Bull. Math.}, 35(5):753--762, 2011.

\bibitem{Anderson1999zero}
D.~F. Anderson and P.~S. Livingston.
\newblock The zero-divisor graph of a commutative ring.
\newblock {\em J. Algebra}, 217(2):434--447, 1999.

\bibitem{atiyah1969introduction}
M.~F. Atiyah and I.~G. Macdonald.
\newblock {\em Introduction to commutative algebra}.
\newblock Addison-Wesley Publishing Co., Reading, Mass.-London-Don Mills, Ont., 1969.

\bibitem{badawi2014annihilator}
A.~Badawi.
\newblock On the annihilator graph of a commutative ring.
\newblock {\em Comm. Algebra}, 42(1):108--121, 2014.

\bibitem{Barkha}
B.~Baloda and J.~Kumar.
\newblock Upper ideal relation graphs associated to rings.
\newblock  To appear {\em{Soft Computing}}, {\em \rm{arXiv}:2403.04266}.

\bibitem{baloda2023study}
B.~Baloda, P.~Mathil, and J.~Kumar.
\newblock A study of upper ideal relation graphs of rings.
\newblock {\em AKCE Int. J. Graphs Comb.}, 21(1):29--40, 2023.

\bibitem{barati2021line}
Z.~Barati.
\newblock Line zero divisor graphs.
\newblock {\em J. Algebra Appl.}, 20(9):Paper No. 2150154, 13, 2021.

\bibitem{beineke1970characterizations}
L.~W. Beineke.
\newblock Characterizations of derived graphs.
\newblock {\em J. Combinatorial Theory}, 9:129--135, 1970.

\bibitem{bera2022line}
S.~Bera.
\newblock Line graph characterization of power graphs of finite nilpotent groups.
\newblock {\em Comm. Algebra}, 50(11):4652--4668, 2022.

\bibitem{MR4543335}
L.~Boro, M.~M. Singh, and J.~Goswami.
\newblock On the line graphs associated to the unit graphs of rings.
\newblock {\em Palest. J. Math.}, 11(4):139--145, 2022.

\bibitem{chakrabarty2009intersection}
I.~Chakrabarty, S.~Ghosh, T.~K. Mukherjee, and M.~K. Sen.
\newblock Intersection graphs of ideals of rings.
\newblock {\em Discrete Math.}, 309(17):5381--5392, 2009.

\bibitem{kelarev2003graph}
A.~Kelarev.
\newblock {\em Graph algebras and automata}, volume 257 of {\em Monographs and Textbooks in Pure and Applied Mathematics}.
\newblock Marcel Dekker, Inc., New York, 2003.

\bibitem{kelarev2009cayley}
A.~Kelarev, J.~Ryan, and J.~Yearwood.
\newblock Cayley graphs as classifiers for data mining: the influence of asymmetries.
\newblock {\em Discrete Math.}, 309(17):5360--5369, 2009.

\bibitem{kelarev2004labelled}
A.~V. Kelarev.
\newblock Labelled {C}ayley graphs and minimal automata.
\newblock {\em Australas. J. Combin.}, 30:95--101, 2004.

\bibitem{khojasteh2022line}
S.~Khojasteh.
\newblock Line cozero-divisor graphs.
\newblock {\em Matematiche (Catania)}, 77(2):293--306, 2022.

 \bibitem{maimani2008comaximal}
 H.~R. Maimani, M.~Salimi, A.~Sattari, and S.~Yassemi.
 \newblock Comaximal graph of commutative rings.
 \newblock {\em J. Algebra}, 319(4):1801--1808, 2008.

\bibitem{kumar2023finite}
Parveen and J.~Kumar.
\newblock On finite groups whose power graphs are line graphs.
\newblock  {\em J. Algebra Appl}, 2024, DOI: 10.1142/S0219498825502858.

\bibitem{pirzada2022line}
S.~Pirzada and A.~Altaf.
\newblock Line graphs of unit graphs associated with the direct product of rings.
\newblock {\em Korean J. Math.}, 30(1):53--60, 2022.

 \bibitem{saha2023prime}
 M.~Saha, A.~Das, E.~Y. \c{C}elikel, and C.~Abdio\u{g}lu.
 \newblock Prime ideal sum graph of a commutative ring.
 \newblock {\em J. Algebra Appl.}, 22(6):Paper No. 2350121, 14, 2023.

\bibitem{singh2022graph}
P.~Singh and V.~K. Bhat.
\newblock Graph invariants of the line graph of zero divisor graph of {$\Bbb{Z}_n$}.
\newblock {\em J. Appl. Math. Comput.}, 68(2):1271--1287, 2022.

\end{thebibliography}
\end{document}